\documentclass[12pt]{amsproc}
\pdfoutput=1
\usepackage[all]{xy}
\usepackage{graphicx}
\usepackage{amssymb}
\usepackage{epstopdf}
\usepackage{amssymb,latexsym}
\theoremstyle{plain}
\newtheorem{theorem}{Theorem}[section]

\newtheorem{example}[theorem]{Example}

\newtheorem{lemma}[theorem]{Lemma}

\theoremstyle{definition}

\theoremstyle{remark}
\newtheorem{remark}[theorem]{Remark}

\begin{document}
\title{Simple wavelet sets in $\mathbb R^n$}
\author[K.
 D.
 Merrill]{Kathy~D.
~Merrill}
\address{Kathy Merrill, Department of Mathematics, Colorado College, Colorado Springs, Colorado, 80903, USA}
\email{kmerrill@coloradocollege.edu}
\begin{abstract}
Wavelet sets that are finite unions of convex sets are constructed in $\mathbb R^n$, $n\geq 2$, for dilation by any expansive matrix that has a power equal to a scalar times the identity and also has all singular values greater than $\sqrt n$.  In particular, we produce simple wavelet sets in every dimension for dilation by any real scalar greater than 1.  \end{abstract}

%beginning
\maketitle
\section{Introduction}
A \textit{wavelet set} relative to dilation by an expansive (all eigenvalues greater than 1 in absolute value) real $n\times n$ matrix $A$ is a set $W\subset \mathbb R^n$ whose characteristic function $\bold 1_W$ is the Fourier transform of an orthonormal wavelet.  
That is, if $\widehat\psi=\bold 1_W,$ then 
$\left\{\psi_{j,k}\equiv\sqrt{|\det A|^j}\psi(A^j\cdot-k),\; j\in\mathbb Z, k\in\mathbb Z^n\right\}$
 is an orthonormal basis for $L^2(\mathbb R^n)$.      This definition is equivalent to the requirement that the
 set $W\subset\mathbb R^n$ tiles $n-$dimensional space (almost everywhere) both under translation by $\mathbb Z^n$ and under dilation by the transpose $A^*$, so that 
 $$\sum_{k\in\mathbb Z^n}\bold 1_W(x+k)=1\quad a.e.\; x\in\mathbb R^n, \mbox{ and}$$
 $$\sum_{j\in\mathbb Z} \bold 1_W(A^{*j}x)=1\quad a.e.\; x\in\mathbb R^n.$$
\par While wavelet set wavelets are not well-localized, and thus not directly useful for applications, they have proven to be an essential tool in developing wavelet theory.  In particular, wavelet set examples established that not all wavelets have an associated MRA \cite{dau}, and that single wavelets exist for an arbitrary expansive matrix in any dimension \cite{dls}.  Smoothing and interpolation techniques have also used wavelet set wavelets to produce more well-localized examples.  (See e.g. \cite{hww}, \cite{hww2}, \cite{dl} , \cite{bs},  \cite{bjmp}, \cite{m3}, \cite{bk}.)

All of the early examples of wavelet sets for dilation by non-determinant 2 matrices in dimension greater than 1 were geometrically complicated, showing the fingerprints of the infinite iterated process used to construct them.  (See e.g. Figure \ref{dim2}(a)).  Many early researchers, e.g.\cite{bs}, \cite{sw}, conjectured  that a wavelet set for dilation by 2 in dimension greater than 1 could not be written as a finite union of convex sets.  In support of this conjecture, Benedetto and Sumetkijakan \cite{bs} showed that a wavelet set for dilation by 2 in $\mathbb R^n$ cannot be the union of $n$ or fewer convex sets.  However, in 2004, Gabardo and Yu \cite{gy} used self-affine tiles to produce a wavelet set for dilation by 2 in $\mathbb R^2$ that is a finite union of polygons (Figure \ref{dim2}(b)).   In 2008 \cite{m1} we used a technique based on generalized multiresolution analyses \cite{hkl} to construct such wavelet sets for arbitrary real ($d>1$) scalar dilations in $\mathbb R^2$.   Figure \ref{dim2}(c) shows one of the wavelet sets for dilation by 2 from \cite{m1}.  Although they were developed independently, and using very different techniques, these two examples are remarkably similar.  In fact, the wavelet sets in Figure \ref{dim2}(b) and \ref{dim2}(c) are equivalent in the sense that one can be transformed into the other under multiplication by a determinant 1 integer matrix.  The similar shape of these two wavelet sets suggests the general $n$-dimensional result produced in this paper.    
 \begin{figure}[h]
\label{dim2}
\centering $\begin{array}{ccc}
   \setlength{\unitlength}{80bp}
\begin{picture}(1,0.6)(0.3,-.2)
\put(0,0){\includegraphics[width=\unitlength]{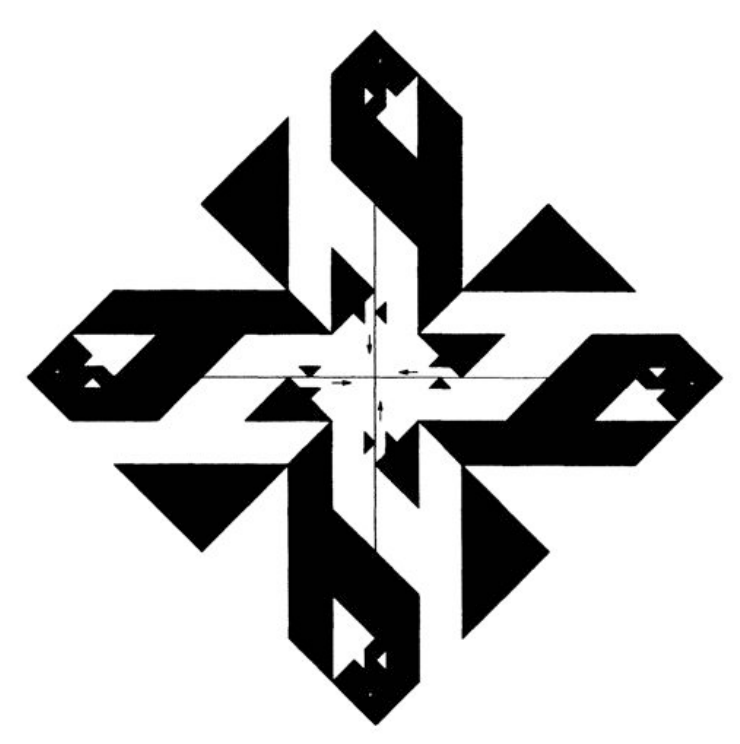}}
\end{picture}
&
\setlength{\unitlength}{100bp}
\begin{picture}(1,1)(.0,-.2)
\put(0,0){\includegraphics[width=\unitlength]{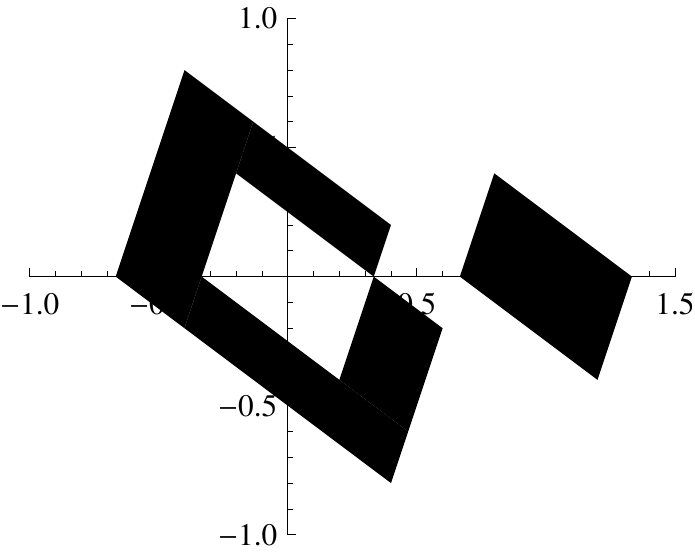}}
\end{picture}
&
 \setlength{\unitlength}{120bp}
\begin{picture}(1,0.0)(-.2,-.2)
\put(0,0){\includegraphics[width=\unitlength]{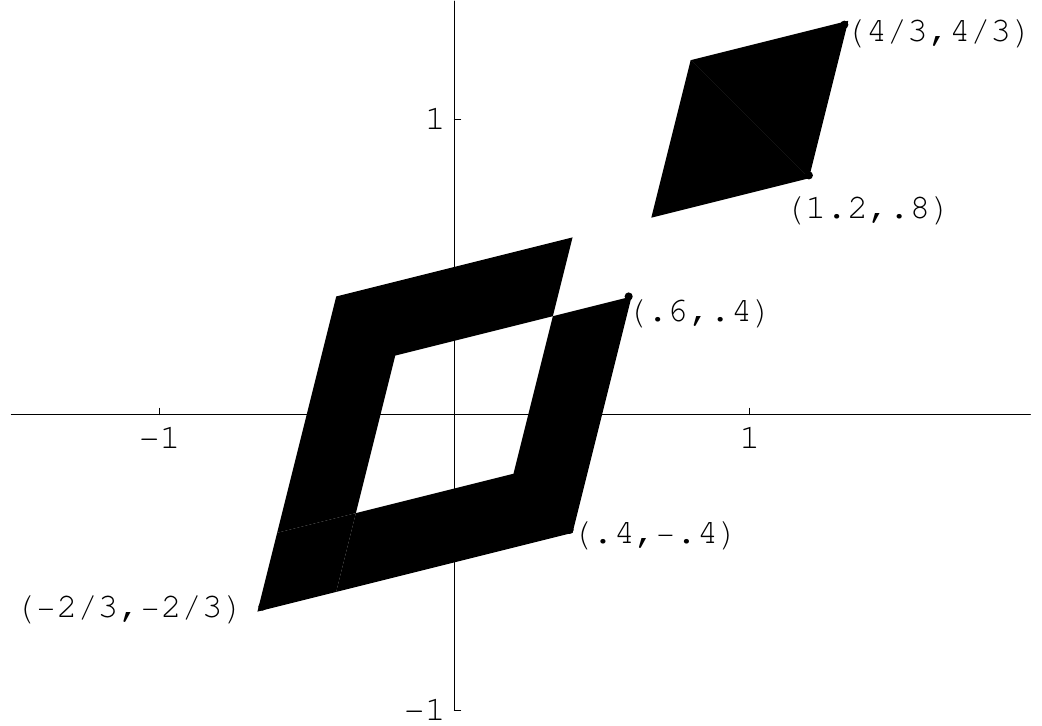}}
\end{picture}\\
\mbox{Soardi/Wieland 1998\quad\quad\quad}&\mbox{Gabardo/Yu 2004}&\mbox{\quad\quad Merrill 2008}\\
\end{array}$
\caption{Wavelet sets for dilation by 2 in $\mathbb R^2$} 
\end{figure}
\par We call wavelet sets that are finite unions of convex sets {\it simple wavelet sets}.  In 2012 \cite{m2}, we expanded the results in \cite{m1} to produce simple wavelet sets for dilation by any $2\times 2$ matrix that has a positive integer power equal to a scalar times the identity, as long as its singular values are all greater than $\sqrt 2$.   In that paper, we also found examples of expansive $2\times 2$ matrices that cannot have simple wavelet sets.  It is our conjecture that, in any dimension, an expansive matrix whose determinant does not have absolute value equal to 2 can have a simple wavelet set if and only if it has a positive integer power equal to a scalar times the identity.  
\par In this paper, we generalize the 2-dimensional examples in \cite{m2} to $n$-dimensional space, $n\geq 2$.   We do this using neither the generalized multi-resolution analysis techniques of \cite{m2}, nor the self-affine techniques of \cite{gy}.  Rather, we use a remarkable result by Sherman Stein \cite{ss} on tiling $\mathbb R^n$ with notched cubes, together with the tiling conditions that are equivalent to the definition of a wavelet set.   Section 2 presents Stein's result, and then skews and translates the notched n-cubes to produce notched parallelotopes that are simple wavelet sets for dilation by negative scalars.  
\par Section 3 further modifies these notched parallelotopes by translating out a central parallelotope (as in Figure \ref{dim2}(b) and \ref{dim2}(c)).  
Using this technique, Theorem \ref{main} creates simple wavelet sets for dilation by any scalar $d\geq 2$.  This result establishes counterexamples, in every dimension greater than 1, to the conjecture that wavelet sets for dilation by 2 cannot be finite unions of convex sets.   These counterexamples are composed of $2n+1$ convex sets for dimension $n$, as compared to the lower bound of $n+1$ given in the Benedetto/Sumetkijakan result mentioned above. Theorem \ref{matrix} generalizes Theorem \ref{main} to dilation by matrices that have a positive integer power equal to a scalar, as long as their singular values are not too small.   One consequence of this theorem is to create  simple wavelet sets for dilation by a scalar $d$ with $1<d<2$, thus  completing the scalar dilation case of the existence question for simple wavelet sets.  For non-scalar dilations in dimension 3 and higher, Theorem \ref{matrix}   
offers support to the sufficiency direction of the conjecture above concerning exactly which matrices have associated simple wavelet sets.  The examples that end Section 3 further support this conjecture by showing that the Theorem's additional condition on singular values need not always hold for matrices that have simple wavelet sets. 

 %%%%%%%%%%%%%%%%%%%%%%%%%%%%%%%%%%%%%
 \section{Tiling with notched parallelotopes}
 
 We begin by establishing some notation.  Write $\{e_1,e_2,\cdots e_n\}$ for the standard basis of $\mathbb R^n$, and $C$ for the cyclic permutation matrix with columns $(e_2,e_3,\dots,e_n,e_1)$.   Let $\vec 1$ stand for the vector $(1,1,\dots,1)\in\mathbb R^n$, and write $\tau_t$ for translation by $t\in\mathbb R^n$.     
 
 Given a vector $v=(v_1,v_2,\cdots,v_n)$ in $\mathbb R^n$ that is not a multiple of $\vec 1$, let 
 $$\mathcal P[v]=\{x_0v+x_1C(v)+\dots x_{n-1}C^{n-1}v\;:\;0\leq x_i\leq 1\}$$
 be the parallelotope spanned by the vectors $\{v,C(v),\dots,C^{n-1} (v)\}$.  Note that the spanning vectors of $\mathcal P[v]$ have equal length, since they are permutations of the same vector $v$.  In particular, when $n=2$, $\mathcal P[v]$ is a rhombus.  Note also that $\mathcal P[v]$ has two vertices on the line determined by $\vec 1$:  one at the origin, and  the other determined by the sum of the coordinates of $v$, at $(\sum v_i)\vec 1$, where the vector   Given an $\alpha$, $0<\alpha<1$, write $\mathcal{N}[v,\alpha]$ for the notched parallelotope that results from deleting a subparallelotope scaled by $\alpha$ from the vertex $(\sum v_i)\vec1$ of $\mathcal P[v]$.  That is, let       
$$\mathcal N[v,\alpha]=\mathcal P[v]\setminus\;\tau_{(1-\alpha)(\sum v_i)\vec 1}\;\alpha \mathcal P[v],$$
or equivalently,
\begin{equation*}
\label{notch}
\mathcal N[v,\alpha]=\mathcal P[v]\setminus\;\tau_{(\sum v_i)\vec 1}\;\left(-\alpha \mathcal P[v]\right).
\end{equation*}

We will need the following result about translation tilings by notched cubes due to Sherman Stein.
 
 \begin{lemma}   \label{notchedcube} 
Given a real number $0<\alpha<1$, let $L$ be the lattice spanned by the columns of $I-\alpha C$, where $C$ is the cyclic permutation matrix.  Then the translates of the notched unit cube $\mathcal {N}[e_1,\alpha]$ by the vectors in $L$ tile $\mathbb R^n$. 
 \end{lemma}
 \begin{proof}
 See \cite{ss}
 \end{proof}
 
We use this result to produce a notched parallelotope that tiles $\mathbb R^n$ under translation by the lattice $\mathbb Z^n$:

 \begin{lemma}   \label{notchedpar}
 For a fixed real number $\alpha$, $0<\alpha<1$, let $w(\alpha)=\frac1{1-\alpha^n}(1,\alpha,\alpha^2,\dots,\alpha^{n-1})$.   Then the translates of $\mathcal {N}[w(\alpha),\alpha]$ by $\mathbb Z^n$ tile $\mathbb R^n$.   
 \end{lemma}
 \begin{proof}
By Lemma \ref{notchedcube} we know that  $\mathcal {N}[e_1,\alpha]$ tiles $\mathbb R^n$ under translation by $L,$ the lattice spanned by the columns of $I-\alpha C$.  If we define $A$ to be the linear transformation that maps $L$ to $\mathbb Z^n$, we thus have that that $A(\mathcal {N}[e_1, \alpha])$ tiles $\mathbb R^n$ under translation by $\mathbb Z^n$.  Note that 
\begin{eqnarray*}
A&=&(I-\alpha C)^{-1}\\
&=&\sum_{i=0}^{\infty}(\alpha C)^i\\
&=&\frac1{1-\alpha^n}\sum_{i=0}^{n-1}(\alpha C)^i,
\end{eqnarray*} 
so that  $A(e_j)
=C^{j -1}w(\alpha)$.   Thus, 
\begin{eqnarray*}
A\left(\mathcal N[e_1, \alpha]\right)&=&A\left(\mathcal P[e_1]\setminus\;\tau_{(1-\alpha)\vec 1}\;\alpha \mathcal P[e_1]\right)\\
&=&\mathcal P[A(e_1)]\setminus\;\tau_{(1-\alpha)A\vec 1}\;\alpha\mathcal P[A(e_1)]\\
&=&\mathcal P[w(\alpha)]\setminus\tau_{\vec 1}\;\alpha \mathcal P[w(\alpha)]\\
&=& \mathcal {N}[w(\alpha), \alpha]
\end{eqnarray*}
 \end{proof}

\par To be a wavelet set, a notched parallelotope would have to also tile under dilation.  For a scalar dilation $d$, this is clearly impossible, since a notched parallelotope is defined to have an extreme point at the origin.  Thus, we consider instead the translated notched parallelotope $\tau_{t}\mathcal N\left[w(\alpha),\alpha\right]$, for $t\in\mathbb R^n$.   For such a set to tile under dilation by $d$ would require that the dilated outer parallelotope, $\frac1d \tau_{t}\;\mathcal P[w(\alpha)]$, fit perfectly into the notch, which is of the form $\tau_{t+\vec 1}\;\alpha \mathcal P[w(\alpha)]$.  First note that this would force the scale of the dilated outer parallelotope to match the scale of the notch, so that $\alpha=|\frac 1d|$.  For positive $d>1$, this perfect fit also would require that $\frac1d t=t+\vec 1$ so that $t=-\frac d{d-1} \vec1$.  However, the outer parallelotope $\tau_{t}\;\mathcal P[w(\frac 1d)]$ has one of its extreme points at $t+\sum_{j=1}^{n}[w(\frac 1d)]_j=t+\frac{d}{d-1}\vec 1$.  Thus, using the required value for the translation $t$ would again cause the outer parallelotope to have an extreme point at the origin. Hence a set of the form $\tau_{t}\mathcal N\left[w(\alpha),\alpha\right]$ cannot tile by a positive scalar dilation.  However, for negative scalar dilations, a wavelet set that is just a notched parallelotope is possible, as the following theorem shows.  The examples produced by Theorem \ref{negdilate} generalize the wavelet sets for negative scalar dilations in $\mathbb R^2$ found in \cite{gy} and \cite{m2}.  

 \begin{theorem}\label{negdilate}
For $d\in\mathbb R$, $d>1$, let $w(\frac 1d)=\frac{1}{d^{n}-1}(d^{n},d^{n-1},\dots,d)$, and  $t=-\frac{d^2}{d^2-1}$.  Then 
$$\mathcal W=\tau_{t\vec 1}\;\mathcal N\left[w\left(\frac 1d\right),\frac 1{d}\right]$$
 is a wavelet set for dilation by $-d$ in $\mathbb R^n$.
\end{theorem}
\begin{proof}
We know from Lemma \ref{notchedpar} that the notched parallelotope $N\left[w(\frac 1d),\frac 1{d}\right]$ tiles $\mathbb R^n$ under translation by $\mathbb Z^n$, and thus that its translate $\mathcal W$ does as well.  It remains to show that $\mathcal W$ tiles under dilation by $-d$. 
 
The proposed wavelet set $\mathcal W=\tau_{t\vec 1}\;\mathcal N\left[w\left(\frac 1d\right),\frac 1{d}\right]$ has its vertices on the line determined by $\vec 1$ at $t\vec 1$ and $(t+1)\vec 1$, while its outer parallelotope $\tau_{t\vec 1}\mathcal P[w\left(\frac 1d\right)]$ has its vertices at $t\vec 1$ and $(t+\frac d{d-1})\vec 1$. (See Figure \ref {neg}a.)  If we consider now the dilate by $-\frac 1d$ of these two polytopes, we see that $-\frac1d \mathcal W$ has its vertices  on the line determined by $\vec 1$ at $\frac{-t}d\vec 1$ and $\frac{-(t+1)}d \vec 1$, while $-\frac1d\tau_{t\vec 1}\mathcal P[w\left(\frac 1d\right)] $ has its vertices at $\frac{-t}d\vec 1$ and $(\frac{-t}d+\frac{-1}{d-1})\vec 1$.  (See Figure \ref{neg}b.)  
\begin{figure}[h]
\label{neg}
\centering $\begin{array}{ll}
   \setlength{\unitlength}{200bp}
\begin{picture}(1,1)(0,0)
\put(0,0){\includegraphics[width=\unitlength]{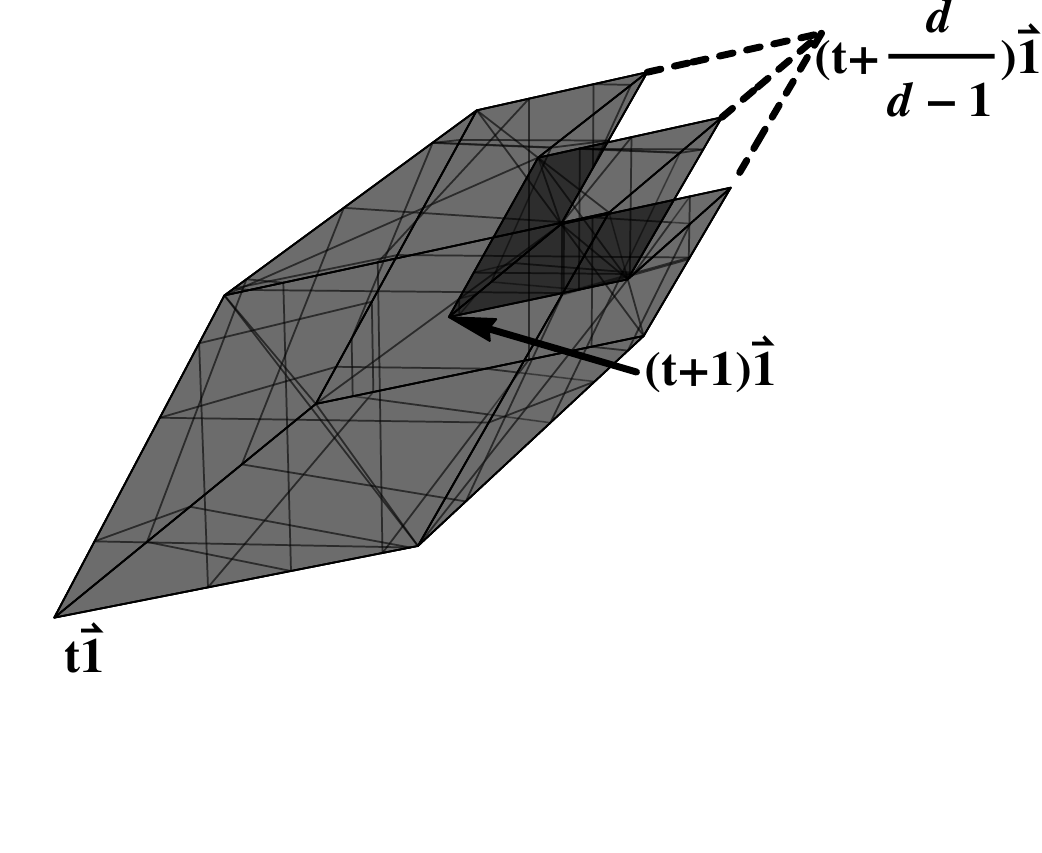}}
\end{picture}
&
\setlength{\unitlength}{200bp}
\begin{picture}(1.0,1)(0,0)
\put(0,0){\includegraphics[width=\unitlength]{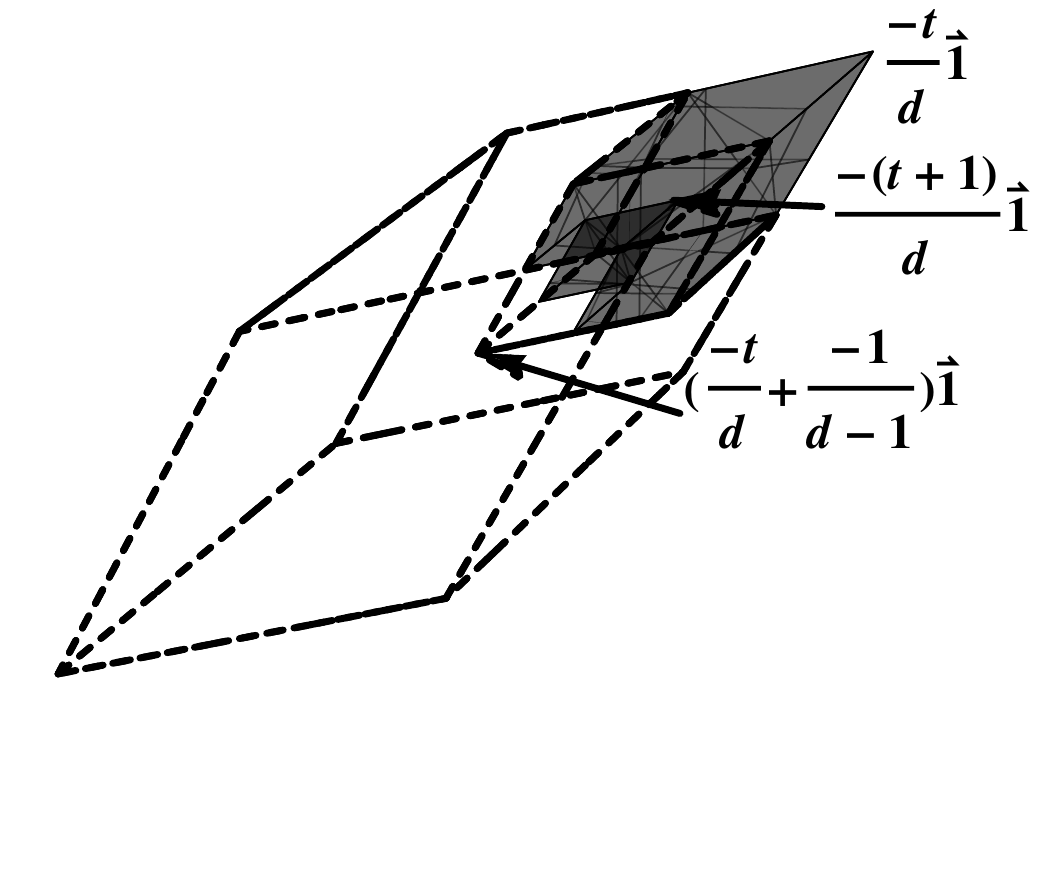}}
\end{picture}\\
\hskip.5in\mbox{(a)}\quad\mathcal W & \hskip.5in \mbox{(b)}\quad -\frac 1d \mathcal W
      \end{array}$
\caption{}       
\end{figure}
Using the value $t=\frac{-d^2}{d^2-1}$ given in the statement of the Theorem, we see that $t+\frac d{d-1}=-\frac td$ and $t+1=\frac{-t}d+\frac{-1}{d-1}$, so that the dilation by $-\frac 1d$ of the outer parallelotope used to form $\mathcal W$ exactly fits into the notch of $\mathcal W$.  Thus,
\begin{eqnarray*}
\mathcal W\sqcup\frac{-1}d\mathcal W&=&\tau_{t\vec 1}\mathcal P\left[w\left(\frac 1d\right)\right]\setminus\left(\tau_{(t+1)\vec 1}\;\frac{1}{d^2}\mathcal P\left[w\left(\frac 1d\right)\right]\right)\\
&=&\tau_{t\vec 1}\mathcal P\left[w\left(\frac 1d\right)\right]\setminus\frac{1}{d^2}\left(\tau_{t\vec 1}\mathcal P\left[w\left(\frac 1d\right)\right]\right),
\end{eqnarray*}
which tiles under dilation by $d^2$.  Therefore, $\mathcal W$ tiles under dilation by $-d$.  
\end{proof} 
Figure \ref{neg}a shows the general shape of a wavelet set $\mathcal W$ for dilation by $-d$ in 3 dimensions.  The size of the notch shown is for $d=2$; in general the size of the notch will be $\frac1d$ times the size of the whole.  For dilation by $-d$, the lower left hand vertex of $\mathcal W$ will be at  $-\frac{d^2}{d^2-1}\vec 1$, the inside corner of the notch at $\frac{-1}{d^2-1}\vec 1$,  and the outer corner of the notch at $\frac{d}{d^2-1}\vec 1.$
%%%%%%%%%%%%%%%%%%%%%%%%%%%
\section{Notched parallelotopes with satellites}
Even though we have seen that a translate of a notched parallelotope $\tau_t\;\mathcal N[w(\alpha),\alpha]$ cannot itself tile under dilation by a positive scalar, we will make a simple alteration to such a set that retains the property of tiling under translation, and makes the set tile under dilation as well.   We use the idea behind the wavelet set construction technique in \cite{bl}.  That is, we eliminate the overlap between our proposed wavelet set and its dilate, by translating the dilate out by an integer vector, creating a satellite.   In order to avoid the iterated process required in \cite{bl}, we translate out a little bigger piece than the dilate of  $\tau_t\;\mathcal N[w(\alpha),\alpha]$; that is, we translate out the dilate of the whole outer parallelotope $\tau_t\;\mathcal P[w(\alpha)]$.  We choose the translation amount such that a dilate of the satellite exactly fills the notch.  The details of this construction are carried out in the following theorem.  
 
 \begin{theorem}\label{main}
For $d\in\mathbb R$, $d\geq 2$, let $w(\frac 1{d^2})=\frac{1}{d^{2n}-1}(d^{2n},d^{2n-2},\dots,d^2)$.  Suppose $k\in\mathbb Z$ satisfies $1\leq k<d$, and let $t=\frac{d(k-d)}{d^2-1}.$  Then
$$\mathcal W=\left(\left( \tau_{t\vec 1}\;\mathcal N\left[w\left(\frac 1{d^2}\right), \frac 1{d^2}\right]\right)\;\setminus\;\left(\frac 1d\tau_{t\vec 1}\; \mathcal P\left[w\left(\frac1{d^2}\right)\right]\right)\right)\;\bigcup\;\tau_{k\vec 1}\left(\frac1d\tau_{t\vec 1}\;\mathcal P\left[w\left(\frac1{d^2}\right)\right]\right)$$  is a wavelet set for dilation by $d$ in $\mathbb R^n.$
\end{theorem}
\begin{proof}

We claim that $\frac1d\tau_{t\vec 1}\;  \mathcal P\left[w\left(\frac1{d^2}\right)\right]\subset\tau_{t\vec 1 }\;\mathcal N\left[w(\frac1{d^2}), \frac 1{d^2}\right]$. 
To see this, first note that
$t<0$ and $-t<\frac{d^2}{d^2-1}=\sum_{i=1}^{n}[w(\frac 1{d^2})]_i,$ so that 
 $ 0\in\frac 1d\tau_{t\vec 1}\;\mathcal P\left[w\left(\frac1{d^2}\right)\right]\subset\tau_{t\vec 1}\; \mathcal P\left[w\left(\frac1{d^2}\right)\right]$.
 Thus, to establish the claim, it will suffice to show that the vertex of the notch that is closest to the origin, namely $(t+(1-\frac 1{d^2})\sum [w(\frac 1{d^2})]_i)\vec 1$, lies outside of   $\frac1d\tau_{ t\vec 1}\;\mathcal P\left[w\left(\frac1{d^2}\right)\right]$.  That is, we must show that $t+1\geq\frac 1d\left(t+\frac{d^2}{d^2-1}\right)$.  Substituting $t=\frac{d(k-d)}{d^2-1},$ we see that this is equivalent to $k\geq\frac 1{d-1},$ which follows from the given conditions $k\geq 1$ and $d\geq 2$.
 
Lemma \ref{notchedpar} impies that $\mathcal N\left[w(\frac1{d^2}),\frac 1{d^2}\right] $ tiles under translation by the integer lattice, and thus that $\tau_{t\vec 1 }\;\mathcal N\left[w(\frac 1{d^2}),\frac 1{d^2}\right]$  does as well.    
By the claim, we have that $\mathcal W$ is formed by removing $\frac 1d\tau_{t\vec 1}\mathcal P\left[w\left(\frac1{d^2}\right)\right]$ from inside $\tau_{t\vec 1 }\;\mathcal N\left[w(\frac 1{d^2}),\frac 1{d^2}\right]$, and shifting it by a vector of the integer lattice.   Thus $\mathcal W$ tiles $\mathbb R^n$ under translation by $\mathbb Z^n$.
 
To establish tiling under dilation by $d$, we first show that $\tau_{ t \vec 1}\;\mathcal N\left[w(\frac 1{d^2}),\frac 1{d^2}\right]$ is disjoint from
$\tau_{k\vec 1}\left(\frac1d\tau_{t\vec 1}\;\mathcal P\left[w\left(\frac1{d^2}\right)\right]\right)$.   That is, we must show that $\frac td+k>t+1$, which follows from the definition of $t$ together with the condition $k\geq 1$.  
Now, note that $\tau_{ t\vec 1}\;\mathcal P\left[w\left(\frac1{d^2}\right)\right]\setminus\frac1d\tau_{t\vec 1} \;\mathcal P\left[w\left(\frac1{d^2}\right)\right]$ tiles $\mathbb R^n$ under dilation by $d$.  The definition of $t$ together with the disjointness of the pieces of $\mathcal W$ shows that $\mathcal W$ is formed from this set by dilating the notch $\tau_{(t+1)\vec 1}\left(\frac1{d^2} \mathcal P\left[w\left(\frac1{d^2}\right)\right]\right)$ by $d$.  Thus,  $\mathcal W$ also tiles $\mathbb R^n$ under dilation by $d$.   

See Figure 9.7 of \cite{m1} for illustrations of these tilings under both translation and dilation in the case $n=2$. 
\end{proof}

\begin{remark}\label{var}The wavelet sets produced by Theorem \ref{main} are a natural generalization of the 2 dimensional simple wavelet sets produced for scalar dilations in \cite{m1}.  We can also alter these examples to produce natural generalizations of the 2-dimensional example for dilation by 2 that appears in \cite{gy}.  Note that if $\mathcal W$ is a wavelet set for dilation by the scalar $d$, then so is $S(\mathcal W)$ for any integer matrix $S$ of determinant $\pm 1$.  If we take $S$ to be the $n\times n$ matrix that has 1's on the diagonal, -1's on the subdiagonal and 0's elsewhere, then $S(\mathcal W)$ is a simple wavelet set for dilation by $d$ in dimension $n$ that is centered on the $x_1$ axis rather than the line $x_1=x_2=\dots x_n$.  Other variations are easily produced using other choices for the matrix $S$.   
\end{remark}

Figure \ref{3d} shows one of the wavelet sets produced by Theorem \ref{main} for dilation by 2 in $\mathbb R^3$, as well as the variation described in Remark \ref{var}.

\begin{figure}[h]
\label{3d}
\centering $\begin{array}{ll}
   \setlength{\unitlength}{150bp}
\begin{picture}(1,1)(.3,0)
\put(0,0){\includegraphics[width=\unitlength]{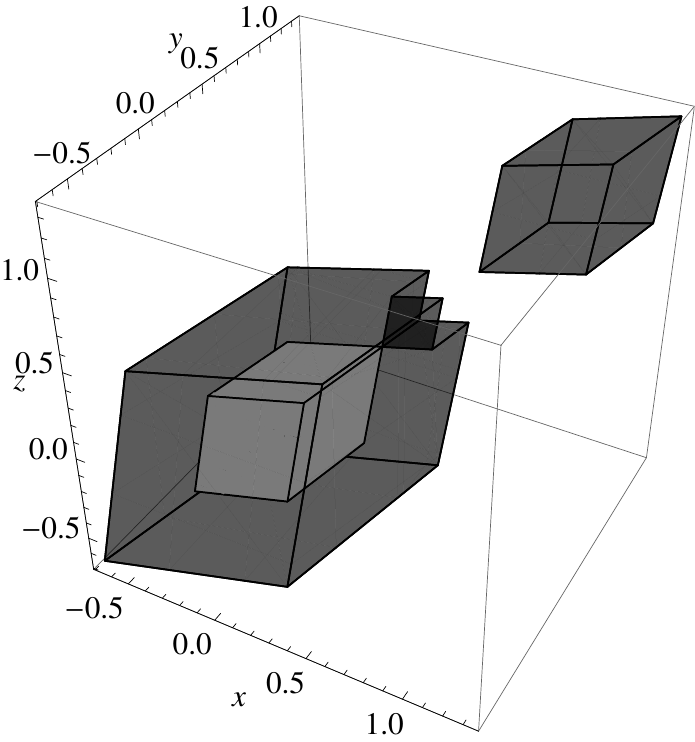}}
\end{picture}
&
\setlength{\unitlength}{150bp}
\begin{picture}(1.0,1)(-.2,0)
\put(0,0){\includegraphics[width=\unitlength]{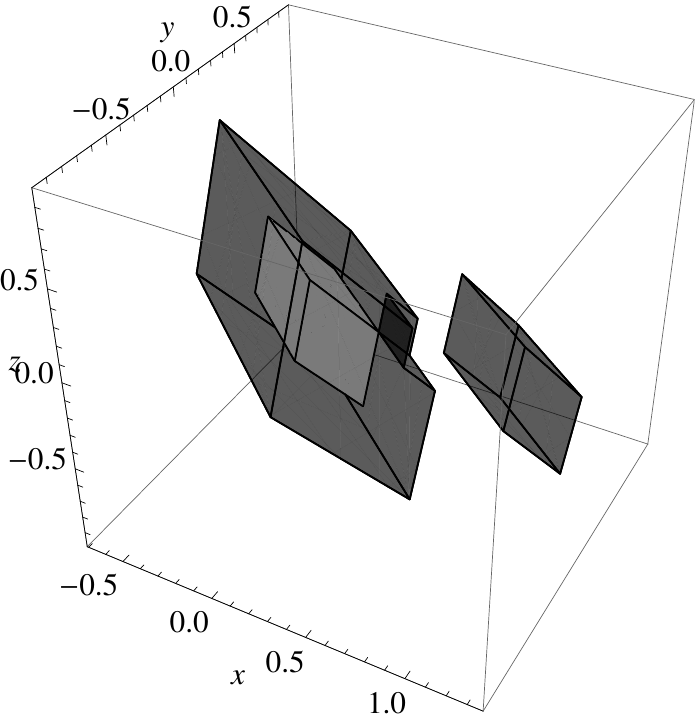}}
\end{picture}
      \end{array}$
      \caption{3-dimensional wavelet sets for dilation by 2.}
\end{figure}

The final theorem uses the same technique as Theorem \ref{main} to produce simple wavelet sets for dilation by an expansive matrix $A$ that has an integer power $A^p$ equal to a scalar multiple $d$ of the identity.  In this case, the notched parallelotope $\tau_t\mathcal N[w(\frac 1d),\frac 1d] $ fails to tile under dilation by $A^*$ because it is formed from the outer parallelotope $\tau_tP[w(\frac1d)]$ by removing its dilate by $\frac 1d$ instead of its dilate by the matrix $A^{*-1}$.  We would like to remedy the resulting overlap of $\tau_t\mathcal N[w(\frac 1d),\frac 1d] $ and all of its negative dilates $A^{*-j}\tau_t\mathcal N[w(\frac 1d),\frac 1d] $, by translating $A^{*-1}\tau_t\mathcal P[w(\frac 1d))]$ out by an appropriately chosen integer vector, to form a satellite that will be clear of $\tau_t\mathcal N[w(\frac 1d),\frac 1d] $ and will perfectly dilate into the notch.  This is possible only if $A^{*-1}\tau_t\mathcal P[w(\frac 1{d^q})] $ is completely contained in $\tau_t\mathcal N[w(\frac 1{d^q}),\frac 1{d^q}] $.  (For an example where this containment does not hold, see Figure 5(a).)  To overcome this difficulty, we put a restriction on $A$, and sometimes replace $\frac 1d$ with a higher power $\frac 1{d^q}$ .

 \begin{theorem}\label{matrix}
Let $A$ be an $n\times n$ integer matrix such that $A^p=d(I[n])$, where $d>1$, $p\in\mathbb Z$, $p>1$, and $I[n]$ is the $n\times n$ identity matrix.  Suppose further that all of the singular values of $A$ are greater than $\sqrt n$.  Let $w(\frac 1{d^q})=\frac{1}{d^{qn}-1}(d^{qn},d^{qn-q},\dots,d^q)$.   Let $k$ be the closest vector in $\mathbb Z^n$ to $A^{*-1}\left(\frac{d^q}2\vec 1\right)$, and let $t=\frac{1}{d^q-1}\left(A^{*}k-d^q\vec 1\right).$   Then for $q$ sufficiently large,  $A^{*-1}\left(\tau_t\mathcal P\left[w(\frac 1{d^q})\right]\right)\subset\tau_t\mathcal N\left[w(\frac 1{d^q}),\frac 1{d^q}\right].$  For such $q,$
$$\mathcal W=\left(\left( \tau_{t}\;\mathcal N\left[w\left(\frac 1{d^q}\right), \frac 1{d^q}\right]\right)\;\setminus\;\left(A^{*-1}\tau_{t}\; \mathcal P\left[w\left(\frac 1{d^q}\right)\right]\right)\right)\;\bigcup\;\left(\tau_{k}\left(A^{*-1}\tau_{t}\;\mathcal P\left[w\left(\frac 1{d^q}\right)\right]\right)\right)$$  is a wavelet set for dilation by $A$ in $\mathbb R^n.$
\end{theorem}
\begin{proof}
By Lemma \ref{notchedpar}, $\mathcal N\left[w(\frac 1{d^q}), \frac 1{d^q}\right]$, tiles under translation by $\mathbb Z^n$, and thus its translate by $t$ does as well.  Thus, $\mathcal W$ will also tile by translation as long as the set $A^{*-1}\tau_{t}\; \mathcal P[w(\frac 1{d^q})]$, which is translated out by the integer vector $k$, is a subset of $\tau_t\mathcal N\left[w(\frac 1{d^q}), \frac 1{d^q}\right]$, and moves to a position disjoint from $\tau_t\mathcal N\left[w(\frac 1{d^q}), \frac 1{d^q}\right]$.  As $q\rightarrow \infty$, the vector $w(\frac1{d^q})\rightarrow e_1$ and $t\rightarrow- \frac{1}2\vec 1$, so that the set $\tau_t\mathcal P\left[w(\frac 1{d^q})\right]$ approaches the unit $n$-cube centered at the origin.  Thus, by taking $q$ sufficiently large, we can make the longest vector in $\tau_t\mathcal P\left[w(\frac 1{d^q})\right]$ to be arbitrarily close to $\sqrt n$ times as long as the shortest vector.  Then, since the singular values of $A$ are greater than $\sqrt n$, we will have $A^{*-1}\left(\tau_t\mathcal P\left[w(\frac 1{d^q})\right]\right)\subset\tau_t \mathcal P\left[w(\frac 1{d^q})\right]$.  The size of the notch also shrinks to 0 as $q\rightarrow \infty$, so that by taking $q$ larger if necessary, we will also have $A^{*-1}\left(\tau_t\mathcal P\left[w(\frac 1{d^q})\right]\right)\subset\tau_t\mathcal N\left[w(\frac 1{d^q}),\frac 1{d^q}\right]$.  By the definition of $k$,  $\tau_{t}\;\mathcal N\left[w\left(\frac 1{d^q}\right), \frac 1{d^q}\right]$ and $\tau_{k}\left(A^{*-1}\tau_{t}\;\mathcal P\left[w\left(\frac 1{d^q}\right)\right]\right)$ are clearly disjoint for large $q$.

To establish tiling under dilation, note that for $q$ large enough that $A^{*-1}\left(\tau_t\mathcal P\left[w(\frac 1{d^q})\right]\right)\subset\tau_t\mathcal N\left[w(\frac 1{d^q}),\frac 1{d^q}\right]$, we have that  $\tau_{ t}\mathcal P[w(\frac 1{d^q})]\setminus\left(A^{*-1}\tau_{t}\; \mathcal P[w(\frac 1{d^q})]\right)$ tiles $\mathbb R^n$ under dilation by  $A^{*}$.   To show that $\mathcal W$ tiles under dilation as well, we must show that the outlier piece
 of $\mathcal W$ exactly fits into the notch of $\tau_{t}\;\mathcal N\left[w(\frac 1{d^q}), \frac 1{d^q}\right]$ under dilation by some integer power of $A^*$.  We have
  \begin{eqnarray*}A^{*(-pq+1)} \left(\tau_{k}\left(A^{*-1}\tau_{t}\;\mathcal P\left[w\left(\frac 1{d^q}\right)\right]\right)\right)&=&A^{*(-pq)}\left(\tau_{A^*k+t}\;\mathcal P\left[w\left(\frac 1{d^q}\right)\right]\right)\\
 &=&\frac 1{d^q}\left(\tau_{A^*k-(d^q-1)t+d^qt}\;\mathcal P\left[w\left(\frac 1{d^q}\right)\right]\right)\\
 &=&\frac 1{d^q}\left(\tau_{d^q\vec1+d^qt}\;\mathcal P\left[w\left(\frac 1{d^q}\right)\right]\right)\\
 &=&\tau_{t+\vec 1}\left(\frac1{d^q}P\left[w\left(\frac 1{d^q}\right)\right]\right),
  \end{eqnarray*} 
  which is exactly the notch of  $ \tau_{t}\;\mathcal N\left[w\left(\frac 1{d^q}\right), \frac 1{d^q}\right]$.
Thus we have that  $\mathcal W$ also tiles $\mathbb R^n$ under dilation by $A^*$.  
\end{proof}
\begin{remark} Theorem \ref{matrix} also produces simple wavelet set for scalar dilations $1<d<2$, which were not covered by Theorem \ref{main}.  For scalar dilations by $d\geq 2$, Theorem \ref{matrix} produces a series of alternative wavelet sets to those of Theorem \ref{main}.  As $q$ increases in this series, the parallelotope becomes closer to cubic, the notch becomes smaller, and the satellite becomes farther removed.  
\end{remark}

\begin{example}
\label{Mat1}
{\rm
Let $A^*=\left(\begin{array}{rrr}
3&0&0\\
0&3&0\\
1&0&-3
\end{array}\right)$.  Then $A^{2}=9I$ and $\mbox{SingularValues}(A)=\{3.54,\; 3, \;2.54\}$, so Theorem \ref{matrix} applies.   
 With $w(\frac19)=(\frac{9^3}{9^3-1},\frac{9^2}{9^3-1},\frac9{9^3-1})$, $k=(1,1,-1)$ and $t=(-\frac34,-\frac34,-\frac58)$, we have $A^{*-1}\tau_t\mathcal P\left[w(\frac 19)\right]\subset\tau_t \mathcal P\left[w(\frac 19)\right]$.  Thus, we have a simple wavelet set 
$$\mathcal W=\left( \tau_{t}\;\mathcal N\left[w\left(\frac 19\right), \frac 1{9}\right]\right)\;\setminus\;\left(A^{*-1}\tau_{t}\;\mathcal P\left[w\left(\frac 19\right)\right]\right)\;\bigcup\;\left(\tau_{(1,1,-1)}\left(A^{*-1}\tau_{t}\;\mathcal P\left[w\left(\frac 19\right)\right]\right)\right),$$
which is pictured in Figure 4. } 
\end{example}

\begin{figure}[h]
\label{matrix1}
\centering 
  \setlength{\unitlength}{200bp}
\begin{picture}(1,1)(.1,0)
\put(0,0){\includegraphics[width=\unitlength]{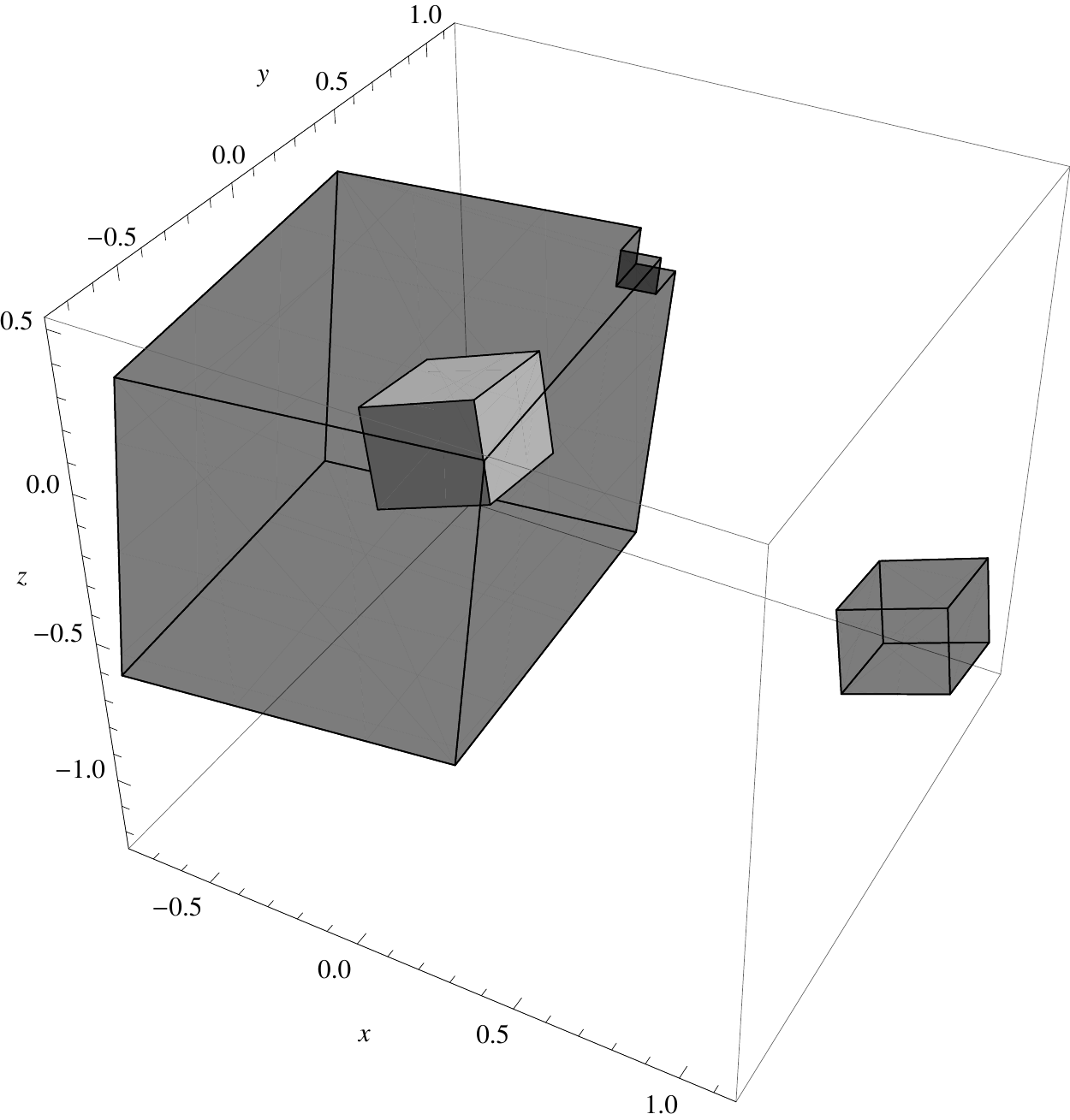}}
\end{picture}
\caption{Simple wavelet set for the matrix $A$ of Example \ref{Mat1}.}
\end{figure}

The hypothesis in Theorem \ref{matrix} that the singular values of $A$ be greater than $\sqrt n$ is sufficient but not necessary, as the next example shows.  
\begin{example}{\rm
\label{Mat}
Let $B^*=\left(\begin{array}{rrr}
2&0&1\\
0&-2&0\\
0&0&-2
\end{array}\right)$.  Then $B^{2}=4I$ and $\mbox{SingularValues}(B)=\{2.56,\; 2, \;1.56\}$, so Theorem \ref{matrix} does not apply.   With $w(\frac 14)=(\frac{64}{63},\frac{16}{63},\frac4{63})$, $k=(1,-1,-1)$ and $t=(-1,-\frac23,-\frac23)$, we do not have $B^{*-1}\tau_t\mathcal P\left[w(\frac 14)\right]\subset\tau_t \mathcal P\left[w(\frac 14)\right]$.  (See Figure  5(a).)  However, using $q=2$, with $w(\frac 1{16})=(\frac{2^{12}}{2^{12}-1},\frac{2^8}{2^{12}-1},\frac{2^4}{2^{12}-1})$, $t=-\frac 8{15}\vec 1$, and $k=(6,-4,-4)$, the required containment does hold, yielding the wavelet set 
$$\mathcal W=\left( \tau_{t}\;\mathcal N\left[w\left(\frac 1{16}\right), \frac 1{16}\right]\right)\;\setminus\;\left(B^{*-1}\tau_{t}\; \mathcal P\left[w\left(\frac 1{16}\right)\right]\right)\;\bigcup\;\left(\tau_{k}\left(B^{*-1}\tau_{t}\;\mathcal P\left[w\left(\frac 1{16}\right)\right]\right)\right).$$
 Figure 5(b) shows the central part of the wavelet set.  (The complete wavelet set also includes a translation of the missing inner parallelotope by $(6,-4,-4)$.) 
\begin{figure}[h]
\label{mat}
\centering $\begin{array}{cc}
   \setlength{\unitlength}{130bp}
\begin{picture}(1,1)(.3,0)
\put(0,0){\includegraphics[width=\unitlength]{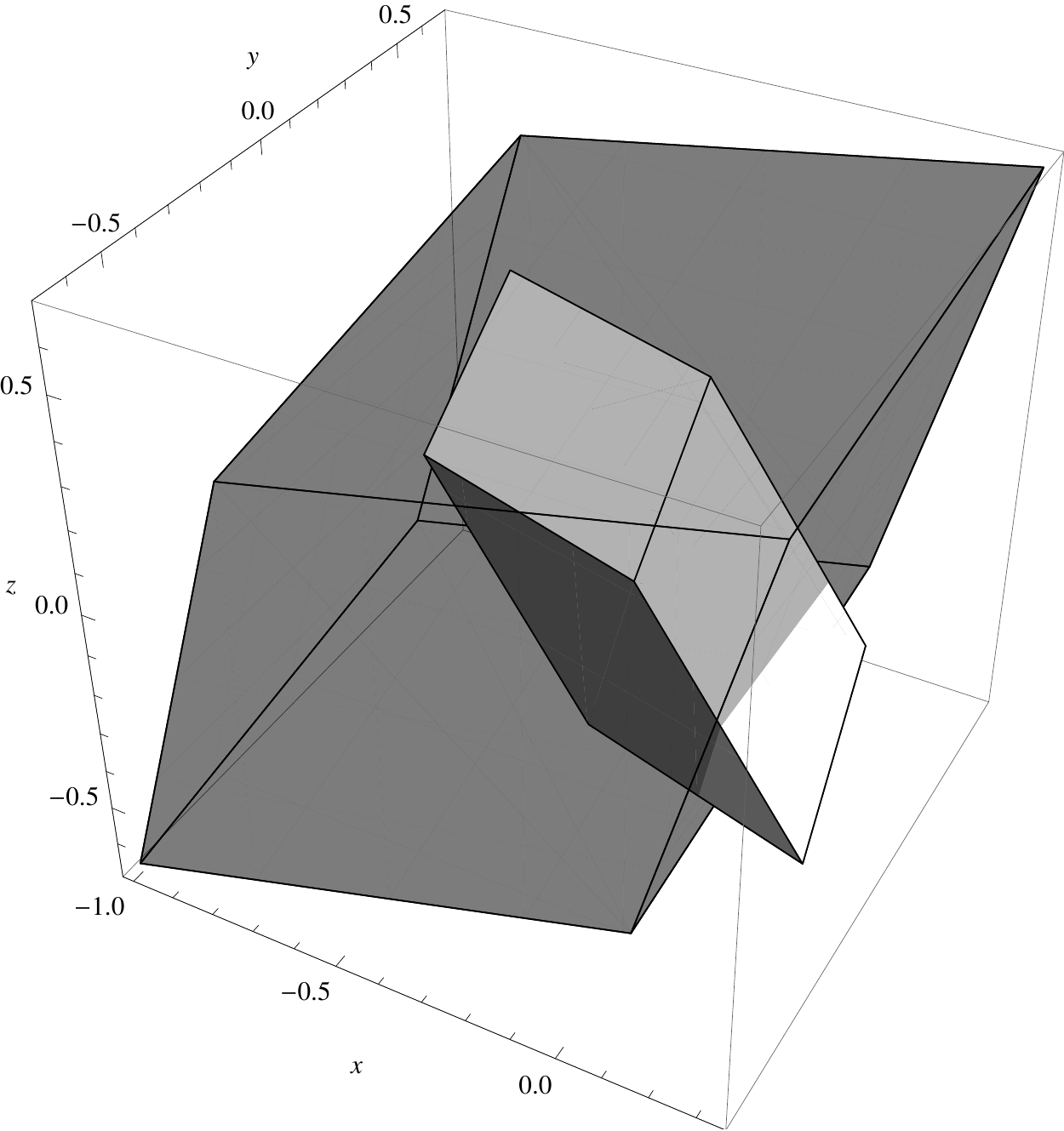}}
\end{picture}
&
\setlength{\unitlength}{130bp}
\begin{picture}(1.0,1)(0,0)
\put(0,0){\includegraphics[width=\unitlength]{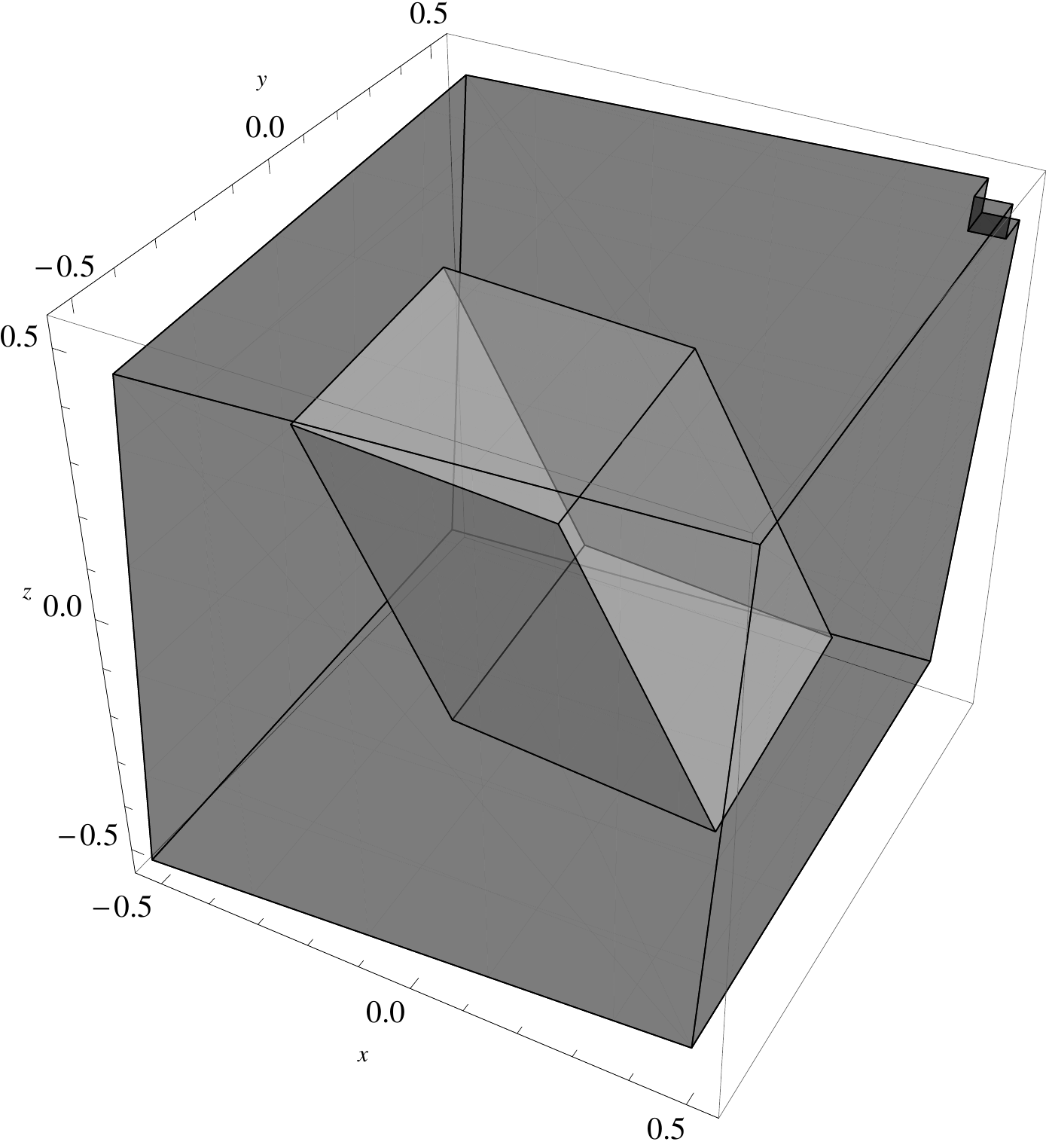}}
\end{picture}\\
\mbox{(a) }B^{*-1}\mathcal P\left[w\left(\frac14\right)\right]\not\subset\mathcal P\left[w\left(\frac14\right)\right]\qquad\qquad&\mbox{(b) }\mathcal N\left[w\left(\frac1{16}\right),\frac1{16}\right]\setminus B^{*-1}\mathcal P\left[w\left(\frac1{16}\right)\right]
      \end{array}$
 \caption{Building a simple wavelet set for the matrix $B$ of Example \ref{Mat}.}

\end{figure}

}\end{example}
%%%%%%%%%%%
\section*{Acknowledgments}
The author wishes to thank the referee for several suggestions that clarified the presentation.

%%%%%%%
\bibliographystyle{amsalpha}

\begin{thebibliography}{99}

\bibitem{bjmp}
L. Baggett, P. Jorgensen, K. Merrill, and J. Packer  (2005).
A non-MRA $C^r$ frame wavelet with rapid decay. \textit{Acta Appl. Math.} 89:251-270.

\bibitem{hkl}
L.~W.Baggett, H.~A.Medina, and K.~D.Merrill (1999). 
Generalized multi-resolution analyses and a construction procedure for all
wavelet sets in $\mathbb{R}^n$. 
\textit{J. Fourier Anal. Appl.} 5:563-573.

\bibitem{bk}
J. J. Benedetto and E.King (2009),
Smooth functions associated with wavelet sets on $\mathbb R^d$ , $d\geq1$, and frame bound gaps,
\textit{Acta Appl. Math.}, 107:121-142.

\bibitem{bl}
J. J. Benedetto and M. T. Leon (1999).
The construction of multiple dyadic minimally supported frequency wavelets on $\mathbb R^d$.
\textit{Contemp. Math.} 247:43-74.

\bibitem{besu}
J. J. Benedetto and S. Sumetkijakan (2006).
Tight frames and geometric properties of wavelet sets.
\textit{Advances in Comp. Math.} 24:35-56.

\bibitem{bs}
M. Bownik and D. Speegle (2002).
Meyer Type Wavelet Bases in $\mathbb R^2$.
\textit{J. Approx. Th.} 116:49-75.

\bibitem{dl}
X. Dai and D. R. Larson (1998).
Wandering vectors for unitary systems and orthogonal wavelets.
\textit{Mem. AMS}134: No. 640.

\bibitem{dls}
X. Dai, D. R. Larson, and D. M. Speegle (1997).
Wavelet sets in $\mathbb R^n$.
\textit{J. Fourier Anal. Appl.} 3:451-456.

\bibitem{dau}
I. Daubechies (1992).
\textit{Ten Lectures on Wavelets}.
American Mathematical Society, Providence RI.

\bibitem{gy}
J-P Gabardo and X. Yu. (2004).
Construction of wavelet sets with certain self-similarity properties.
\textit{J. Geom. Anal.} 14: 629-651.


\bibitem{hww}
E. Hern\'{a}ndez, X. Wang, and G. Weiss (1996).
Smoothing minimally supported frequency wavelets I.
\textit{J. Fourier Anal. Appl.} 2:329-340.

 \bibitem{hww2}
E. Hern\'{a}ndez, X. Wang, and G. Weiss (1997).
Smoothing minimally supported frequency wavelets II.
\textit{J. Fourier Anal. Appl.} 3:23-41.

\bibitem{m3}
K.~D.~Merrill (2008)
Smooth well-localized Parseval wavelets based on wavelet sets in $\mathbb R^2$.
\textit{Contemp. Math.} 464: 161-175.

\bibitem{m1}
K.~D.~Merrill (2008).
Simple wavelet sets for scalar dilations in $L^2(\mathbb R^2)$.
In, \textit{Wavelets and Frames: a Celebration of the Mathematical Work of Lawrence Baggett} (P. Jorgensen, K. Merrill and J. Packer eds.), Birkhauser, Boston, pp.177-192. 

\bibitem{m2}
K.~D.~Merrill (2012).
Simple wavelet sets for matrix dilations in $\mathbb R^2$ .
\textit{Num. Funct. Anal. Opt.} 33:1112-1125.  
 
\bibitem{sw}
P. M. Soardi and D. Weiland (1998). 
Single wavelets in n-dimensions.
\textit{J. Fourier Anal. Appl.} 4:299-315.

\bibitem{ss}
S.K. Stein (1990).
The notched cube tiles $\mathbb R^n$.
\textit{Discrete Math.}80:335-337.   


\end{thebibliography}

\end{document}